\documentclass[12pt,reqno]{amsart}

\usepackage{titling}

\usepackage{amsaddr}

\usepackage[usenames]{color}
\usepackage[latin1]{inputenc}

\usepackage{enumerate}
\usepackage{amssymb, amsmath}
\usepackage{geometry,graphicx}
\usepackage{hyperref}
\usepackage{tikz-cd}
\newtheorem*{theorem*}{Theorem}
\newtheorem*{remark*}{Remark}
\theoremstyle{plain}
\newtheorem{theorem}{Theorem}

\newtheorem{lemma}[theorem]{Lemma} % lemmas, corollaries, etc., are numbered consecutively
\newtheorem{corollary}[theorem]{Corollary}
\newtheorem{proposition}[theorem]{Proposition}
\newtheorem{conjecture}[theorem]{Conjecture}

\theoremstyle{definition}
\newtheorem{remark}[theorem]{Remark}

\newtheorem{definition}[theorem]{Definition}
\newtheorem{example}[theorem]{Example}
\newtheorem{examples}[theorem]{Examples}

\DeclareMathOperator{\Gal}{Gal}
\DeclareMathOperator{\GL}{GL}
\DeclareMathOperator{\SL}{SL}
\DeclareMathOperator{\Aut}{Aut}

\DeclareMathOperator{\sign}{sign}

\DeclareMathOperator{\rad}{rad}
\begin{document}
%	\date{\today}
%title
\title{Solving $a x^p + b y^p = c z^p$ with $abc$ containing an arbitrary number of prime factors}


%authors
\author{Luis Dieulefait}
\thanks{The first author is partially supported by MICINN grant MTM2015-66716-P}

\address[A1]{Departament de Matem\`atiques i Inform\`atica\\ 
Universitat de Barcelona, \\
Gran Via de les Corts Catalanes, 585, 08007, Barcelona, Spain}

\author{Eduardo Soto}
\thanks{The second author is partially supported by MICINN grant MTM2016-78623-P}

%Adresses
\address[A2]{Departament de Matem\`atiques i Inform\`atica\\ 
Universitat de Barcelona,\\
Gran Via de les Corts Catalanes, 585, 08007, Barcelona, Spain}

\email[A1]{ldieulefait@ub.edu}
\email[A2]{edusoto91@gmail.com}
\begin{titlingpage}

\maketitle

%abstract
\begin{abstract}
	In this paper we prove new cases of the asymptotic Fermat equation with coefficients.
	This is done by solving some remarkable $S$-unit equations and applying a method of 
	Frey-Kraus-Mazur.
\end{abstract}
\end{titlingpage}
%INTRODUCTION

\section*{Acknowledgements} 
	We would like to thank Samuele Anni, Henri Cohen, Nuno Freitas, Roberto Gualdi, 
	Xavier Guitart, Mariagiulia De Maria, Artur Travesa, Carlos de Vera and Gabor Wiese for 
	helpful conversations and comments. The second author is very grateful to Marc Masdeu 	and Alberto 
	Soto  for their help on computational aspects.
	We would like to thank the anonymous referees for a thorough reading of our paper, and for the 
	numerous helpful suggestions they made to improve the exposition.

\section*{Introduction}
	Let $p$ be a rational prime and consider the degree $p$ Fermat equation
		\begin{equation}\label{eq}
			x^p +  y^p +  z^p = 0.
		\end{equation}
	The group $\mathbb Q^\times$ acts on the set of rational solutions 
	of \eqref{eq} by 
	\[
	\lambda(x,y,z) = (\lambda x, \lambda y, \lambda z), \qquad \lambda \in \mathbb Q^\times.
	\]
	That allows us to consider solutions  in the rational projective plane 
	\[
	\mathbb P_2(\mathbb Q) = (\mathbb Q^3\setminus \boldsymbol{0})/ \mathbb Q^\times,
	\]
	That is, equation \eqref{eq} defines a projective plane curve $F_p$ in $\mathbb P_2$.
	
	By the \emph{genus-degree} formula $F_p$ has genus 
	\[
		g_p=(p-1)(p-2)/2.
	\]
	Faltings' theorem \cite{Faltings} states that
	the set $F_p(\mathbb Q)$ of $\mathbb Q$-rational points of $F_p$ is finite if $g_p\geq 2$.
	Genus $0$ and genus $1$ curves, corresponding to $p=2$ and $p=3$ respectively, 
	might have infinitely many rational points.
	The main goal in this paper is to prove a finiteness statement hence, we will
 	avoid the case $p\leq 3$. 

	Fermat's last Theorem predicted that 
	\[
	F:= \bigcup_{p\geq 5} F_p(\mathbb Q) =\{[1:-1:0], [1:0:-1], [0:1:-1]\}.
	\]
	In this paper we are interested in the finiteness of $F$ and we shall generalize it to Fermat equations with coefficients. 
	Let $a, b, c$ be non-zero integers and let $F^{a, b, c}_p$ denote the projective curve given by
	\[
	ax^p + b y^p + c z^p=0.
	\]
	The Asymptotic Fermat Conjecture with coefficients $a, b, c$ predicts that 
		\begin{conjecture}\label{conjecture}
			The set
			\[
				AF_{a, b, c} := \bigcup_{p\geq 5} F^{a, b, c}_p(\mathbb Q) 
			\]	
			is finite.
		\end{conjecture}
	It is straightforward to see that the set of \emph{trivial points} in $AF_{a, b, c}$, i.e. points $[x:y:z]$ satisfying 
	$xyz=0$, is finite.	
			
	The very first non-trivial evidence of Conjecture \ref{conjecture} was established by Andrew 
	Wiles when proving Taniyama-Shimura conjecture for the semistable case and hence 
	proving the case $a=b=c=1$.

		\begin{theorem*}[Wiles \cite{Wiles}]
			\[AF_{1,1,1} = \{[1:0:-1],[1:-1:0],[0:1:-1]\}.\]
		\end{theorem*}
	
	\begin{remark*}
		Case $p=3$ of Fermat's last Theorem was proved by Leonhard Euler. 
	\end{remark*}
	
	Jean-Pierre Serre, Barry Mazur and Gerhard Frey had previously established some cases of 
	the conjecture, conditionally on  Serre's conjecture or Taniyama-Shimura conjecture; both proved now.
	
	\begin{theorem*}[Serre \cite{Serre87}]\label{serrethaf}
		Let $n$ be a non-negative integer and let $q$ be a prime in 
		\[
			\{3,5,7,11,13,17,19,23,29,53,59\}.
		\]
		Then 
		\[
			AF_{1,1,q^n} \subseteq F^{1,1,q^n}_5(\mathbb Q) \cup F^{1,1,q^n}_7(\mathbb Q) \cup  
			F^{1,1q^n}_q(\mathbb Q)\cup \{\text{trivial points}\}.
		\]
	\end{theorem*}

	\begin{theorem*}[Frey-Mazur  \cite{Frey}]
		Let $q$ be an odd prime which is neither a Mersenne prime nor a Fermat prime, let $n$ be a positive integer 
		and $m$ a 
		non-negative integer.
 		Then 
		\[
		AF_{ 1,q^n,2^m}\qquad \text{is finite.}
		\] 
	\end{theorem*}
	Kenneth Ribet, for the case $2\leq m <p$, and Henri Darmon, Loïc Merel, for the 
	case $m=1$ studied equation $X^p + Y^p + 2^m Z^p=0$. In particular they proved that
	\begin{theorem*}[Ribet \cite{Ribet}, Darmon-Merel \cite{DarmonMerel}]
	     	\[
	     	AF_{1,1,2^m} = \{[1:-1:0]\} \cup \{[2^r: 2^r:-1] | \text{ $m = rp+1$ for
		$p\geq 5$}\}.
		\]
	\end{theorem*}
	
	For a non-zero integer $N$ let $\rad(N)$ denote the greatest square-free divisor of $N$. 
	Let $\textbf{P}$ denote the set of prime numbers. We can and will identify the image of
	$\rad:\mathbb Z\setminus \{0\}\rightarrow \mathbb N$ with the set of finite subsets of 
	$\textbf{P}$. In particular the \emph{radical} of $\pm1$ corresponds to the empty set 
	under that identification. Similarly $\rad'(N)$ will denote the greatest odd divisor
	of $\rad(N)$.
	
	Alain Kraus has given effective bounds related to the Asymptotic Fermat conjecture and 
	proved the following.
	
	\begin{theorem*}[Kraus {\cite[Corollaire 1]{Kraus}}]\label{krausth}
		Let $(a, b, c)$ be non-zero pairwise coprime integers such that $\rad(abc) = 2q$ for
		an odd prime $q$ which is neither a Mersenne prime nor a Fermat prime. 
		Then there is an explicit constant $G=G(a, b, c)$ such that 
		\[
			AF_{a, b, c}=\{\text{trivial points}\} \cup\bigcup_{5\leq p<G} F^{a, b, c}_p(\mathbb Q).
		\]
	\end{theorem*}
	\begin{remark*}
		Case $(a, b, c)= (1,1,2^\alpha{q^\beta})$ is not explicitly stated in Kraus' paper. 
		Nevertheless, the same method as for the case  $(1,2^\alpha,q^\beta)$ applies.
	\end{remark*}

	Related to this conjecture Nuno Freitas, Emmanuel Halberstadt and
	Alain Kraus, have recently developed the so-called \emph{symplectic method} to solve 
	Fermat equations for a positive density of exponents $p$, see \cite{FreitasKraus} or
	 \cite{HalKraus}. 
	Our approach follows similar strategies as in \cite{Kraus} and relies strongly on modularity;
	see \cite{CohenII} chapter 15 for an exposition of the modular method written by Samir 
	Siksek and Theorem 15.5.3 therein for an improvement of Serre's Theorem, \ref{serrethaf}.

	In this paper we exhibit non-trivial local obstructions\footnote{We use the terminology
	\emph{non-trivial local obstructions} to distinguish from the ones introduced in Proposition 
	\ref{trivialobs}.} to some $S$-unit equations and we deduce 
	results as the following. Let $(a, b, c)$ be a primitive tern, i.e. $gcd(a, b, c)=1$, of non-zero integers.

	\begin{theorem}
		 Assume that $\rad(abc)$ is a product of primes all in $1+12\mathbb Z$ then
			$AF_{a,b,2^rc}$
		is finite for every $r\geq 0$, $r\neq 1$.
	\end{theorem}
	
	\begin{theorem}
		Assume that $\rad(abc)$ is a product of primes all in $1+3\mathbb Z$ then 
		$AF_{a,b,16c}$ is finite.
	\end{theorem}
	
	We also consider some particular cases with $\rad(abc)=q\ell$, for different odd
	primes $q,\ell$. For example
		
	\begin{theorem}
		Let $q$, $\ell\geq 5$ be primes such that $q\equiv -\ell\equiv 5\pmod{24}$.
		If $\rad(abc) = q\ell$ then, $AF_{a, b, c}$ is finite.
	\end{theorem}
		See section \ref{statements} for the complete list of cases we consider.
	\begin{remark*}
		We use Kraus method to deduce explicit bounds $G(a, b, c)$ on $p$, see 
		Section \ref{sectionbounds}.
	\end{remark*}			
%	\begin{remark}
%		Kraus paper assumes $a, b, c$ being pairwise coprime. 
%	\end{remark}	
	
\numberwithin{theorem}{section}

\section{Fermat-type curves}\label{FermatType}
	Let $a, b, c\in \mathbb Z$, $p$ prime, $abc\neq 0$. 
	By a Fermat-type curve we mean a projective plane curve of the form
	\[
		F_p^{a, b, c}: a x^p + b y^p + c z^p=0.
	\]
	Notice that the Fermat-Type curve $F_p^{a,b,c}$ is a twist of the classical one $x^p+y^p + z^p=0$
	thus, they share some geometric properties as the genus. Also, the condition
	$abc\neq 0$ is equivalent to $F_p^{a,b,c}$ being non-singular.
	
	\begin{theorem}[Faltings, \cite{Faltings}]
		Let $C/\mathbb Q$ be a projective curve of genus $\geq 2$. Then $C(\mathbb Q)$ is finite. 
	\end{theorem}
	
	By the genus-degree formula one has that $F_p^{a,b,c}/\mathbb Q$ has genus
	\[
		(p-1)(p-2)/2.
	\]
	This is a consequence of Hurwitz theorem, \cite[II, 5.9]{Silverman}.\footnote{
	Consider the degree $p$ morphism $\phi: F_p^{a, b, c} \rightarrow \mathbb P_1$, $[x:y:z]\mapsto [x:y]$. 
	It is ramified at $p$ points with constant ramification index $p$.} 
	Thus $F_p^{a, b, c}(\mathbb Q)$ is finite for $p\geq 5$. 
	The sets $F_2^{a, b, c}(\mathbb Q)$, $F_3^{a, b, c}(\mathbb Q)$ might be infinite\footnote{The
	set $F_2^{a, b, c}(\mathbb Q)$ is infinite if and only if it is not empty. 
	If $\mathcal O\in F_3^{a, b, c}(\mathbb Q)$ then $(F_3^{a, b, c},\mathcal O)$ is an elliptic curve
	over $\mathbb Q$ and $F_3^{a, b, c}(\mathbb Q)$ is a finitely generated group.}.
	
	Let  $a, b, c$ be non-zero integers and let us  consider the set $AF_{a, b, c}$ defined in
	Conjecture \ref{conjecture}. The following result is a direct consequence of 
	\cite[Proposition 1.1]{DarmonGranville}.
	
	\begin{proposition}\label{trivialobs}
		Assume that there is a prime $\ell$ such that $v_\ell(a)$, $v_\ell(b)$, $v_\ell(c)$ are pairwise different.
		Then $AF_{a, b, c}$ is finite.
	\end{proposition}
	\begin{proof}
		Let us see that 
		\[
			F_p^{a, b, c}(\mathbb Q)=\emptyset
		\]
		for every $p > k:=\max(v_\ell(a),v_\ell(b),v_\ell(c))$.
		Let $[x:y:z]\in F_p^{a,b,c}(\mathbb Q)$ and 
		\[
			(A, B, C)=(ax^p, b y^p, c z^p)\neq (0,0,0).
		\]
		Then $v_\ell(A) \leq v_\ell(B)\leq v_\ell(C)\leq \infty$ 
		up to permutation of $A, B, C$ and $A+B+C =0$. 
		Notice that 
		\[
			v_\ell(A)= v_\ell(B)< \infty
		\]
		since $v_\ell(A) = v_\ell(B + C) \geq v_\ell(B)$ and $(A,B,C)\neq (0,0,0)$. 
		Hence
		\[
			v_\ell(a)\overset{\pmod p}{\equiv} v_\ell(A) = v_\ell(B) \overset{\pmod p}{\equiv} v_\ell(b)
		\]
		Thus $v_\ell(a) = v_\ell(b)$ and 
		\[
			AF_{a, b, c} = \bigcup_{5\leq p\leq k}F_p^{a, b, c}(\mathbb Q) 
		\]
		is finite.
	\end{proof}
	
	We say that a tern $a, b, c$ has a \emph{trivial local obstruction} if there is a prime $q$
	such that $v_q(a)$, $v_q(b)$, $v_q(c)$ are pairwise different. Thus, we shall
	focus on terns with no trivial obstruction. We make the following hypothesis.
	\[
		\boldsymbol{(F)}: \text{The tern $(a, b, c)$ has no trivial local obstruction and $gcd(a, b, c)=1$}.
	\]
	Notice that $a, b, c$ satisfies $\boldsymbol{(F)}$ if $a, b, c$ are pairwise coprime.
	
	\begin{lemma}\label{lema1.3}
		Let $a, b, c$ be a tern satisfying $\boldsymbol{(F)}$ and let $p$ be a prime such that
		\[
			p > \max_{\text{$q$ prime}} \max(v_q(a),v_q(b),v_q(c)).
		\]

		Then there are pairwise coprime integers $\alpha,\beta,\gamma$ such that
		$\rad(\alpha\beta\gamma)= \rad(abc)$ and
		\[
			F_p^{a, b, c}\simeq F_p^{\alpha,\beta,\gamma}
		\]
		as algebraic curves over $\mathbb Q$. 
	\end{lemma}
	\begin{proof}
		Let $a, b, c$ be non-zero integers satisfying $\boldsymbol{(F)}$ and let 
		\begin{align*}
			T_a 	&= \gcd(b,c),\\
			T_b 	&= \gcd(a,c),\\
			T_c 	&= \gcd(a,b).		
		\end{align*}
		Then there are integers $a',b',c'$ such that $a',b',c',T_a,T_b,T_c$ are pairwise coprime
		and 
		\[
			\begin{array}{ccc}
				a = a' 	&T_b	&T_c,\\
				b = b' 	&T_a	&T_c,\\
				c = c' 	&T_a	&T_b.
			\end{array}
		\]
		The Lemma follows by an induction on the number of prime divisors of $T_aT_bT_c$.
		Assume that $1\leq e:= v_q(T_a) < p$. 
		The linear map $[x:y:z]\mapsto [qx:y:z]$ defines an isomorphism 
		$F_p^{a_1,b_1,c_1}\rightarrow F_p^{a, b, c}$, where $q^e(a_1,b_1,c_1)=(q^pa, b, c)$. 
		Hence 
		\[
			\begin{array}{ccc}
				T_{a_1}	&=	&T_a/q^e,\\
				T_{b_1}	&=	&T_b,\\
				T_{c_1}	&=	&T_c.
			\end{array}
		\]
		and $\rad(a_1b_1c_1) = \rad(abc)$. 
		
	\end{proof}
	\begin{remark}\label{technicalvaluation}
		With the notation above one has that $v_q(\alpha\beta\gamma) = v_q(abc)$
		if $q\nmid T_aT_bT_c$ and $v_q(\alpha\beta\gamma) = p-v_q(abc)/2 =p-v_q(T_aT_bT_c)$ 
		otherwise.
	\end{remark}
\section{\texorpdfstring{$S$}{S}-unit equations}\label{sunitsection}
	Let $S$ be a finite set of primes. We identify $S$ with its product in 
	$\mathbb Z$.
	Let $a, b, c$ be non-zero integers and consider the projective line 
	$L:aX + b Y + c Z=0$ attached to it. The set
	\[
		L(\mathbb Q)=\{[x:y:z]\in \mathbb P^2(\mathbb Q): L(x,y,z)=0\}
	\]
	of $\mathbb Q$-rational points of $L$ is infinite. 
	\begin{definition}
	Let $P\in L(\mathbb Q)$ and let $(x,y,z)\in\mathbb Z^3$ be a primitive representative of 
	$P$, that is $gcd(x,y,z)=1$. We say that $P$ is an \emph{$S$-point} of $L$ if $xyz\neq 0$ and 
	$\rad(xyz)\mid S$. 
	\end{definition}
	\begin{theorem}[Siegel-Mahler]
		Let $a, b, c$ be non-zero integers and let $S$ be a finite set of prime 
		numbers.
		The set of $S$-points in the line
		\[
			L: aX + b Y + c Z=0
		\]
		is finite.
	\end{theorem}

	\begin{proof}[Proof (Lang {\cite[p. 28]{Lang}} )]
		The $S$-points of $L$ correspond to the points of the affine curve 
		$C: aX + b Y + c=0$ with values in 
		\[
		\Gamma=\mathbb Z[1/S]^\times < \mathbb Q^\times.
		\] 
		The set $S$ together with $-1$ generate the abelian group 
		$\Gamma$, hence $\Gamma/\Gamma^5$ is finite. 
		If $C$ has infinitely many points with coefficients in $\Gamma$, 
		then infinitely many points $\{(x_i,y_i)\}_{i\geq 1}$ coincide mod $\Gamma^5$.
		Thus the curve $ax_1 X^5 + b y_1 Y^5 + c=0$ has infinitely many 
		rational points and it has genus $6$ since $ax_1by_1c\neq0$. This contradicts 
		Faltings' theorem.
	\end{proof}
%	By \emph{solving the $S$-unit equation $aX+ bY + c Z$} we mean finding all $S$-unit points in the line.
	
	Let us focus on the projective line 
	\[
		L_0:X+Y+Z=0.
	\]
	Frey-Kraus-Mazur (FKM) method on the Asymptotic Fermat Conjecture with coefficients $(a, b, c)$ considers the set
	of primes $S=\rad(2abc)$ and seeks for $S$-points
	of $L_0$. The $2$-adic valuation of the $S$-points will play an important role. 
	\begin{definition}
	Let $S$ be a set of primes and let $P\in L_0(\mathbb Q)$ be a $S$-point
	with primitive representative $(x,y,z)\in \mathbb Z^3$. We say that $P$ is a 
	\emph{proper $S$-point} if $\rad(P)=S$.
	The \emph{height} $h_2(P)$ is defined by $h_2(P)=v_2(xyz)$.
	\end{definition}
	One has that 
		\[
			1\leq  h_2(P)<  \infty
		\]
	for every $P\in L_0(\mathbb Q)$. 
	In particular $L_0$ has no $S$-points if $2\notin S$.
	\begin{example}
	For $S=2$ the $S$-point of $L_0$ is $[2:-1:-1]$ up to permutation of coordinates.
	\end{example}
	\begin{proposition}\label{2.5}
		Let $q$ be an odd prime and let $S=2q$. Then the set of proper $S$-points of $X+Y+Z=0$
		is
		\begin{itemize}
		\item 
		$\{[3:-2:-1],[3:-4:1], [9:-8:-1]\}$, if $q=3$,
		\item
		$\{[q:-2^n:-1]\}$ with height a power of $2$, if $q$ is a Fermat prime $\geq 5$,
		\item 
		$\{[q:-2^n:+1]\}$ with prime height, if $q$ is a Mersenne prime $\geq 7$,
			\item
		$\emptyset$, otherwise,
		\end{itemize}
		up to permutation of coordinates.
	\end{proposition}
	\begin{proof}
		Let $(x,y,z)$ be a primitive representative of a proper $S$-point.
		Then $x,y,z$ are non-zero pairwise coprime integers 
		We may assume without loss of generality that $x=q^m, y=-2^n, z=\pm 1$, $m,n\geq 1$.
		By Theorem \ref{partialcatalan} one has that either $q=n=3$ and $m=2$
		or $m=1$. We deduce that case $q=3$ has $3$ points. Case $m=1$, $q\geq 5$ implies that either 
		$q=2^n+1$ is a Fermat prime and $n$ is a power of $2$ or $q=2^n-1$ is a Mersenne prime
		and $n$ is prime. Notice that $3$ is the only prime being Fermat and Mersenne.
	\end{proof}
	
%	\subsection{Proper $S$-unit points and elliptic curves}		
%		\begin{lemma}
%		Let $N$ be an odd square-free integer and let $s\in \{0,1,3,5\}$. 
%		Let 
%		$$
%			H=\begin{cases}
%				\{4\},							&\text{if $s=0$,}\\
%				\mathbb N\setminus\{0,1,2,3,4\},	&\text{if $s=1$,}\\
%				\{2,3\},						&\text{if $s=3$,}\\
%				\{1\},							&\text{if $s=5$.}
%			\end{cases}
%		$$
%		The following are equivalent
%		\begin{enumerate}
%			\item There is an elliptic curve $E/\mathbb Q$ with full rational 
%			two-torsion and conductor $2^sN$.
%			\item There is a proper $2N$-unit point of height in $H$. 
%		\end{enumerate}
%	\end{lemma}
%	
%
%	
%	

	The FKM method to Conjecture \ref{conjecture} relies on finding pairs $(S,H)$ such that
	\begin{itemize}
		\item $S$ is a finite set of primes containing $2$,
		\item $H$ is a set of non-negative integers to be defined and
		\item there is no proper $S$-point of height $h_2\in H$ in $X+Y+Z=0$. 
	\end{itemize}
	In the following subsections we exhibit infinite families of such pairs.
	%, see Theorem \ref{mod24} and section \ref{largeS}.
	
\subsection{\texorpdfstring{$S=2q\ell$}{S=2ql}}
	 
Let $q,\ell$ be odd primes.
In this subsection we deal with equations of the form
\[
	\begin{array}{ccccc}
		2^r q^s 	&= 	&\ell^t 		&\pm 	&1\\
%		2^r \ell^t 	&= 	&q^s 		&\pm 	&1\\
		2^r 		&= 	&q^s \ell^t 	&\pm 	&1\\
		2^r 		&= 	&q^s 		&\pm 	&\ell^t \\
%		2^r 		&= 	&\ell^t 		&-		&q^s.
	\end{array}
\]
	For a non-zero integer $k$ let $\sigma(k)$ denote the number of divisors of 
	$n$ and let $\omega(k)$ denote the number of prime divisors of $k$.
	Let $\Phi_k$ denote the $k$th cyclotomic polynomial. See Appendix \ref{appA}
	for further details on cyclotomic polynomials. 

\begin{proposition}\label{completesol}
	Let $\ell,q$ be odd primes and assume that
	\[
		2^r q^s = \ell^{2t}-1
	\]
	for some positive integers $r,s,t$. The solutions are given by the following equalities
	\begin{align*}
		2^4 \cdot 5 &= 3^4 - 1\\
		2^3 \cdot3 &= 5^2 -1,\\
		2^4 \cdot 3 &= 7^2 -1,\\
		2^5 \cdot 3^2 &= 17^2-1.
	\end{align*}
\end{proposition}
\begin{proof}
	Let $r,s,t,q,\ell$ be a solution. 
	Notice that $3\mid \ell q$ since either $\ell\mid 3$ or $3\mid \ell^{2t}-1$. 

	If $\ell=3$ then $\sigma(2t) \leq 3$ by Corollary \ref{corminus}.
	Hence $t\in \{1,2\}$ with solutions $3^2-1=2^3$, $3^4-1 = 2^4\cdot 5$.
	The case $3^2-1=2^3$ is not allowed.
	
	If $q=3$ then $t=1$ by Corollary \ref{corminus}.
	The integers $\ell+1$, $\ell-1$ are consecutive even numbers and 
	$\gcd(\ell+1,\ell-1) = 2$. Case $\ell\pm 1 = 2^{r-1}\cdot 3^s$, 
	$\ell\mp 1=2$ is not possible since $\ell\geq 5$. So assume 
	\begin{align*}
		\ell+\varepsilon &=2\cdot 3^s\\
		\ell -\varepsilon &=2^{r-1}
	\end{align*}
	for some unit $\varepsilon$. Then $3^s - 2^{r-2} = \varepsilon$
	with solutions $3 -2 = 1$, $3-2^2=-1$, $3^2 - 2^3 = 1$ by Proposition  \ref{2.5}.
	Hence $\ell\in \{5, 7, 17\}$. 
\end{proof}

\begin{proposition}\label{1505}
	Let $\ell,q$ be odd primes and assume that
	\[
		2^r q^s = \ell^{t}-1
	\]
	for some positive integers $r, s, t$ such that $t$ is odd and $\geq 3$. Then 
	$t$ is prime, $\Phi_1(\ell)=\ell -1=2^r$ and $\Phi_t(\ell)=q^s$. 
	Hence $\ell$ is a Fermat prime. 
\end{proposition}
\begin{proof}
	The odd integer $t$ is prime since $ 2\leq \sigma(t)\leq \omega(\ell^t-1) =2$ by Corollary \ref{corminus}. 
	Thus 
	\[
		2^r q^s =(\ell-1)\Phi_t(\ell)
	\]
	by the polynomial factorization in \eqref{factors1}.
	Notice that $\Phi_t(\ell)$ is odd and has a prime divisor coprime to $2t$ by Theorem \ref{1805}, then
	$q\mid \Phi_t(\ell)$ and $q\neq t$. Notice that the greatest common divisor
	of $\ell-1$ and $\Phi_t(\ell)$ divides $t$.
	Hence $\ell-1$ and $\Phi_t(\ell)$ are coprime.
\end{proof}

		\begin{proposition}\label{technicallemma3}
			Let $\ell,q$ be odd primes and assume that
			\[
				2^r q^s = \ell^{t}+1
			\]
			for some positive integers $r,s,t$ such that $t$ is odd and $\geq 3$. 			
			Then $t$ is prime, $\ell+1 = 2^r$ and $\Phi_{2t}(\ell) = q^s$.
			Hence $\ell$ is a Mersenne prime.
		\end{proposition}
		\begin{proof}
			The integer $t$ is prime since $2\leq \sigma(t) \leq \omega(\ell^t+1)=2$ by Corollary \ref{corplus}. Thus
			\[
				2^r q^s = (\ell+1)\Phi_{2t}(\ell)
			\]
			by the factorization of \eqref{factors} in Appendix \ref{appA}.
			One proves as in Proposition \ref{1505} that $\Phi_{2t}(\ell)$ and $\ell+1$ are coprime. 
			Notice that $\Phi_{2t}(\ell)$ is odd.
		\end{proof}

		\begin{proposition}\label{technicalproposition1}
			Let $\ell,q$ be odd primes and assume that
			\[
				2^r q^s = \ell^{2t}+1
			\]
			for some positive integers $r,s,t$. 			
			Then $r=1$ and $2t=2^m$ for some $m\geq 1$ and $2^{m+1}\mid q-1$. 
		\end{proposition}
		\begin{proof}
		Let $t_2$ be the largest odd divisor of $2t$. Then 
		$\ell^{2t}+1$ has $\geq \sigma(t_2)$ odd prime divisors by Corollary \ref{corplus}. 
		Thus $t_2=1$. 
%		 then Corollary \ref{corplus} 
%		says that $t_2$ is either $1$ or an odd prime.
%		That is, either $X^{2t}+1=\Phi_{2^{m+1}}$ or 
%		$X^{2t}+1 = \Phi_{2^{m+1}} \Phi_{p2^{m+1}}$ for some odd prime $p$. 
%		The integer $\Phi_{2^{m+1}} \Phi_{p2^{m+1}}(\ell)$ has two odd prime divisors
%		by Theorem \ref{primedivisors} thus $2t= 2^m$.
		In particular $\ell$ has order $2^{m+1}$ in $\mathbb F_q^\times$.
		Notice that $\ell^{2^m}+1\equiv 2 \pmod4$ thus $r=1$. 
		\end{proof}

			Let $n\geq 2$ be an integer. 
			The equation $2X^n - 1 = Z^2$ was studied by Carl Störmer 
			in \cite[Section 3]{Stormer}.
			He proved that either $n$ is a power of two or $X=Z=1$ is
			the only solution in $\mathbb Z$. See also \cite[A11.1]{Ribenboim}. 
			
		\begin{proposition}\label{technicalthm2}
			Assume that there are odd primes $\ell, q$ such that 
			\[
				2^r q^s = \ell^{2t} +1
			\]
			for some integers $r, s,t\geq 1$. 
			Then $r=1$ and either 
			\begin{itemize}
				\item $s=1$ and $q = \dfrac{\ell^{2t} +1}{2}$, or
				\item $(q,\ell) = (13, 239)$, or
				\item $s=2, t=1$.
			\end{itemize}
		\end{proposition}
		\begin{proof}
			We have already seen in Proposition \ref{technicalproposition1} that $r$ is necessarily $1$. 
			Case $s=1$ has many solutions.
			Assume $s\geq 2$. Then $s$, $2t$ are powers of two due to Störmer's result
			and Proposition \ref{technicalproposition1}. The curve $2x^4 = y^2 + 1$ has two
			positive integer solutions $(1,1)$ and $(13,239)$ due to \cite{Ljunggren} or \cite{SteinerTzanakis}.
			For the study of $C: 2x^2 = y^4 +1$, we consider the rational map $\varphi: C\dashrightarrow E$
			$$
				(x,y)\mapsto \left(\dfrac{4x}{x-y} - 2, \dfrac{4-4x^2}{(x-y)^2}\right)
			$$
			where $E$ denotes the elliptic curve given by
			$y^2 = x^3 + 4 x$ with Cremona Label 32a1. 
			Notice that $\varphi$ is well defined in the Zariski open $U= C\setminus\{x=y\}$ and 
			that $\varphi$ maps rational points of $U$ to rational points of $E$.
			The Mordell-Weil group of $E$ consists in $4$ points, they are $(0,0)$, $(2,4)$, $(2,-4)$ and the 
			point at infinity, see \cite{LMFDB}. The computation of $\varphi^{-1}E(\mathbb Q)$ and 
			$C\cap \{ x = y\}$ provides the equality $C(\mathbb Q) = \{\pm(1,-1), \pm(1,1)\}$.

			Thus equation $2 q^{2^n} = \ell^{2^m} + 1$ has only solution $(q,\ell)=(13,239)$ for $mn\geq 2$. 
		\end{proof}

	\begin{remark}
		Let $q$, $\ell$ be odd primes and assume that $\ell$ is neither a Fermat prime nor a Mersenne
		prime. One deduces from previous statements an algorithm to determine the solutions 
		to the equation $2^r q^s = \ell^t+\varepsilon$. Indeed, the cases $s=1$ or $t=1$ 
		are easy to deal with. 
		Let us assume that $s, t\geq 2$. The case $\varepsilon =-1$
		is completely treated in Propositions \ref{completesol} and \ref{1505}.
		The case $\varepsilon=1$ and $t$ odd is solved in Proposition \ref{technicallemma3}.
		For the case $\varepsilon =1$ and $t$ even one has by Proposition \ref{technicalthm2} that either 
		$(q,\ell) = (13, 239)$ or $s=t=2$ and $r=1$.
		
		One can use elementary algebraic number theory to attack the equation $\ell^2 - 2q^2 =-1$. 
		Notice that an integer point of $x^2 - 2 y^2 =-1$ corresponds to the unit 
		$x+y\sqrt 2\in \mathbb Z[\sqrt 2]^\times$.
		Thus, all these points arise as powers of the fundamental
		unit $\eta=1+\sqrt 2$, i.e. $\mathbb Z[\sqrt 2]^\times = \{\pm1\} \cdot \eta^\mathbb Z $.
		Four such pairs $(\ell,q)$ arise as coefficients of $\eta^n$, with $3\leq n\leq 10^4$. 
	\end{remark}
	
	There are indeed solutions to equation $2^r q^s = \ell^t\pm 1$. 
		
	\begin{examples}
		\begin{itemize}
			\item $2\cdot 5^2 = 7^2 + 1$ corresponding to $\eta^3$,
			\item $2\cdot 29^2 = 41^2 + 1$ corresponding to $\eta^5$,
			\item $\eta^{29}$,
			\item $\eta^{59}$,
			\item $2\cdot 13^4 = 239^2 + 1$,
			\item $\Phi_5(3) = 11^2$ and $2 \,\Phi_5(3) = 3^5-1$,
			\item $\Phi_7(5)$ is prime and $2^2\cdot  \Phi_7(5) = 5^7 -1$,
			\item $\Phi_{34}(7)$ is prime and $2^3\cdot \Phi_{34}(7) = 7^{17}+1$. 
		\end{itemize}
	\end{examples}
	
	We finish this subsection with a general statement about $2S$-unit equations for $|S|=2$. 
		
	\begin{lemma}\label{mod24}
		Let $q$, $\ell\geq 5$ be primes.
		Assume one of the following:
		\begin{enumerate}
			\item $(q,\ell) \equiv (-5, 5)$ or $(11,-11) \pmod {24}$.
			\item $q\equiv 11 \pmod{24}$, $\ell\equiv 5\pmod{24}$ and
			$\genfrac (){}{}{q}{\ell} = -1$.
			\item $q\equiv \pm 3 \pmod 8$, $\ell\equiv -1 \pmod {24}$, $\ell\not 
			\equiv -1
			\pmod {q}$.
		\end{enumerate}
%		Let $a, b, c$ pairwise coprime integers such that $\rad(abc) = q_1q_2$.
		Then the $2\ell q$-unit equation 
		\[
			 X + Y + Z = 0
		\]
		has no proper points of height $\geq 3$.
	\end{lemma}
	\begin{proof}
		This is a mod $24$ exercise. See Appendix \ref{mod24ex}.
	\end{proof}	
	\subsection{Large \texorpdfstring{$|S|$}{|S|}}\label{largeS}
	\begin{lemma}[$h_2=4$]\label{r4}
		Let $S$ be a finite set of primes in $1+ 3\mathbb Z$.
		Then $L_0$ has no $2S$-points of height $4$.
	\end{lemma}
	
	\begin{proof}
		Let $(A, B, C)$ be a (primitive representative of a) $2S$-point of height $4$
		and let $\varepsilon_A,\varepsilon_B,\varepsilon_C$ the sign
		of $A, B, C$, respectively. Then 
		\[
			0=A+B+C \equiv \varepsilon_A +\varepsilon_B+\varepsilon_C\pmod 3.
		\]
		Hence $(\varepsilon_A, \varepsilon_B,\varepsilon_C) = \pm (1,1,1)$ and $A+B +C$ 
		is either strictly positive or strictly negative.
	\end{proof}	
	
	Notice that the same proof applies to every even height case. 
	In particular, $L_0$ has no $2S$-point of even height with the notation of Lemma \ref{r4}.
	
	\begin{lemma}[$h_2=4$]\label{r41}
		Let $n$ be a positive integer not dividing $14$, $16$ nor $18$ and let 
		$S$ be a finite set of primes in $\pm 1+ n\mathbb Z$. Then $L_0$ has 
		no $2S$-points of height $4$.
	\end{lemma}
	\begin{proof}
		Let $(A, B, C)$ be a $2S$-point of height 4. Say $A=2^4 A'$, then $A',B, C \equiv \pm 1 \pmod n$. 
		Thus 
		\[
		0=A+B+C\equiv \pm 16 \pm 1 \pm 1\pmod n.
		\]
		Hence $n\mid 14, 16$ or $18$.
	\end{proof}
	\begin{lemma}[$h_2\geq 2$]\label{r2}
		Let $p$ be an odd prime. Let $S$ be a finite set of primes 
		in $1+4p\mathbb Z$. Then $L_0$ has no $2S$-points 
		of height $\geq 2$.
	\end{lemma}
	\begin{proof}
		Let $(A, B, C)$ be a proper point. Say
		\[
		\begin{array}{ccccc}
			A	&=	&\varepsilon_A 		&A'		&2^r\\
			B	&=	&\varepsilon_B 		&B'\\
			C	&=	&\varepsilon_C 	&C'
		\end{array}
		\]
%		\begin{align*}
%			A&=\varepsilon_A A' 2^{r}\\
%			B&=\varepsilon_B B'\\
%			C&=\varepsilon_C C'
%		\end{align*}
		for $r\geq 2$ and $\varepsilon_x = \sign x$. 
		Then $A'\equiv B'\equiv C' \equiv 1\pmod{4p}$ and 
		\[
		0 = A + B + C \equiv 2^r \varepsilon_A+\varepsilon_B + \varepsilon_C \pmod{4p}.
		\]
		Thus $\varepsilon_B\equiv -\varepsilon_C\pmod 4$ and $ 2^r \varepsilon_A+\varepsilon_B + \varepsilon_C\equiv 0\pmod p$. 
		Then $\varepsilon_B =-\varepsilon_C$ and $p\mid 2^r$. 
	\end{proof}

	\section{Frey-Kraus-Mazur method}
		In this section we recall the FKM method. The standard references are Frey's \cite{Frey} and
		Kraus' \cite{Kraus} papers.
		Let $a, b, c$ be non-zero pairwise coprime integers and let 
		\begin{equation}\label{primebound}
			p>\max \left(4, \max_{\text{$q$ prime}} v_q(abc)\right)
		\end{equation}
		be a prime. Assume that $F_p^{a, b, c}(\mathbb Q)$ has a non-trivial point $P$ and let $(x,y,z)$ 
		be a primitive tern of non-zero integers such that $[x:y:z] =P$. That is, 
		$xyz\neq 0$, $gcd(x,y,z)=1$  and 
		\[
			ax^p +by^p + cz^p=0.
		\]
		Notice that $(A, B, C)=(ax^p,by^p,cz^p)$ are pairwise coprime integers.\footnote{
		Indeed, if $q\mid x,y$ then $q^p\mid c$ and $p\leq v_q(c)\leq v_q(abc)$.}
	
	\subsection{The Frey curve} 

	Following the notation above consider the elliptic curve
	\[
		E=E_{A, B, C}:Y^2 = X(X-A)(X+B)
	\]
	over $\mathbb Q$. 
	The definition of $E_{A, B, C}$ is sensible to the order of $(A, B, C)$. More precisely,
	the curve $E_{A, B, C}$ is a twist of $E_{B,A,C}$ by the quadratic twist of $\mathbb Q(i)/\mathbb Q$ while even permutations
	of $(A, B, C)$ define $\mathbb Q$-isomorphic elliptic curves. Hence $E_{A, B, C}$, $E_{B,A,C}$
	have common prime-to-$2$ conductor.
	Let us reorder $(A, B, C)$ so that $E$ has minimal conductor exponent over 
	$\mathbb Q_2$.\footnote{For example one can take $B$ even and $A\equiv-1\pmod 4$.} 
	
	\begin{proposition}\label{toprove2}
		$E$ has conductor $2^r \rad'(abcxyz)$ where 
		\[
		r = \left\{
			\begin{array}{ccl}
				1	&	&\text{if $xyz$ is even or $v_2(abc)\geq 5$,}\\
				0	&	&\text{if $xyz$ is odd and $v_2(abc) =4$,}\\
				3	&	&\text{if $xyz$ is odd and $v_2(abc) \in\{2,3\}$,}\\
				5	&	&\text{if $xyz$ is odd and $v_2(abc) =1$.}\\
			\end{array}
		\right.
		\]
	\end{proposition}
	\begin{proof}
		The elliptic curve $E$ has 
		semi-stable reduction at every odd prime since $A, B, C$ are pairwise coprime. 
		Let $\ell$ be an odd prime, then $E$ has bad reduction over $\mathbb Q_\ell$
		if and only if $\ell\mid ABC$. 
		Thus, $E$ has  prime-to-$2$ 
		conductor $\rad'(ABC)=\rad'(xyzabc)$ by Neron-Ogg-Shafarevich. The conductor exponent of $E$ over $\mathbb Q_2$
		has been computed in \cite[Lemma 2]{DiamondKramer}. If $xyz$ is even
		then $v_2(ABC)\geq p v_2(xyz) \geq p > 4$ by hypothesis, thus $r=1$.
		Notice that $v_2(abc)=0$ implies $xyz$ even.
	\end{proof}

	\begin{lemma}\label{irredlemma}
		The Galois representation
		\[
			\bar{\rho}_{E,p} :\Gal(\bar{\mathbb Q}/\mathbb Q)\longrightarrow \Aut(E[p])\simeq \GL_2(\mathbb F_p)
		\]
		is irreducible. 
	\end{lemma}
	\begin{proof}
		Recall that $p\geq 5$ by assumption \eqref{primebound}. 
		The irreducibility condition is proved in Serre's paper \cite[Proposition 6]{Serre87} for the 
		semi-stable case, i.e. $r\leq1$. Serre's proof relies on Mazur's theorem
		\cite[Theorem 2]{Mazur}. 
		
		Let us prove irreducibility for $r\in\{3,5\}$. 
		Consider the local Galois representation
		\[
			\rho_{E,p}\vert_{G_2}: \Gal(\bar{\mathbb Q}_2/\mathbb Q_2)\rightarrow \Aut(\mathcal T_p(E))
		\]
		and the residual representation
		\[
			\bar{\rho}_{E,p}\vert_{G_2}: \Gal(\bar{\mathbb Q}_2/\mathbb Q_2)\rightarrow \Aut(E[p]),
		\]	
		where $\mathcal T_p(E)$ denotes the $p$-adic Tate module of $E$ and $G_2$ denotes
		a decomposition subgroup of $G_{\mathbb Q}$ over $2$. 
		The conductor of $\rho_{E,p}\vert_{G_2}$ is larger than or equal to the conductor of
		$\bar{\rho}_{E,p}\vert_{G_2}$. Henri Carayol computed in \cite{Carayol89} the cases
		where the inequality is strict. See the discussion in page 789 and Proposition 2 therein. 
		Since $\rho_{E,p}\vert_{G_2}$ has unramified determinant
		and $r\geq 3$ one deduces that $\rho_{E,p}\vert_{G_2}, \bar{\rho}_{E,p}\vert_{G_2}$
		have common conductor $2^r$. Assume that $\bar{\rho}_{E,p}$ is reducible then
		\[
		\bar{\rho}_{E,p}\vert_{G_2}\simeq \left(
		\begin{array}{cc}
			\chi_1	&*\\
					&\chi_2
		\end{array}
		\right)
		\]
		with $\chi_1\chi_2$ being the (unramified) mod $p$ cyclotomic character. 
		Thus $\chi_1$, $\chi_2$ have common conductor. The Swan conductor is invariant under 
		semisimplification. Thus, the Swan conductor of $\bar{\rho}_{E,p}\vert_{G_2}$ coincides with the
		Swan conductor of $\chi_1\oplus \chi_2$. That is, 
		either $\chi_1$, $\chi_2$ are unramified and $\bar{\rho}_{E,p}\vert_{G_2}$ has conductor $\leq 1$
		or $\chi_1,\chi_2$ are ramified with common Swan conductor $m$. In the last case one has 
		that $\bar{\rho}_{E,p}\vert_{G_2}$ has even conductor exponent 
		$r=\dim_{\mathbb F_p} E[p] - \dim_{\mathbb F_p} E[p]^{I_2} + 2m = 2+2m$.
		
		\end{proof} 
		
		\subsection{Lowering the level}
		
		We shall lower the level of $E$ via $E[p]$. The standard reference here is Ribet's
		level lowering theorem, \cite{RibetLow}. Let us recall some notation therein. Let $\bar\rho:=E[p]$
		be the mod $p$ irreducible representation attached to $E$. Then $\bar\rho$ is modular of
		level $N=2^r\rad'(abcxyz)$ by Wiles \cite{Wiles}, see also \cite{DiamondKramer}.
		Let $\ell$ be a prime divisor of $N$ with $\ell \parallel N$, that is $\ell \mid N$ and $\ell^2 \nmid N$. The representation $\bar\rho$ is finite at $\ell$
		if by definition some geometric condition is satisfied.\footnote{More precisely, $\bar\rho$ is finite at $\ell$
		if there is a finite flat $\mathbb F_p$-vector space scheme $H$ over $\mathbb Z_\ell$ such that  
		$H(\bar{\mathbb Q}_\ell)$ is isomorphic to $\bar\rho\vert_{G_\ell}$ as $\mathbb F_p[G_\ell]$-modules.}
		For the case of modular elliptic curves that condition has a pleasant equivalence.
		\begin{lemma}
			Let $p$ be a prime, let $E'$ be an elliptic curve over $\mathbb Q$ of conductor $N'$ and let $\ell\parallel N'$
			be a prime. Then $E'[p]$ is finite at $\ell$ if and only if $p\mid v_\ell(j_{E'})$. If $p\neq \ell$ then $E'[p]$ is finite
			at $\ell$ if and only if $E'[p]$ is unramified at $\ell$. 
		\end{lemma}
		\begin{proof}
			The lemma is a consequence of Tate's uniformization for multiplicative reduction elliptic curves over $\mathbb Q_\ell$. 
			See \cite[2.12]{DDT} and \cite[8.2]{Edixhoven}.
		\end{proof}
		Let $s$ be the  conductor exponent of $E[p]$ at $2$, $s\leq r$. If $r\in\{0,3,5\}$ then $s=r$. 
		If $r=1$ then $s$ is ruled by Tate's uniformization. That is, $s=0$ if and only if $v_2(abc)=4$. 
		\begin{theorem}\label{newform}
			Following the notation above, let 
			\[
				\bar{\rho}:\Gal(\bar{\mathbb Q}/\mathbb Q)\rightarrow \GL_2(\mathbb F_p)
			\]
			be the Galois representation attached to the $p$-torsion of $E=E_{A, B, C}$.
			There is a newform $f\in S_2(2^s\rad'(abc))$ whose mod $p$ Galois representation
			is isomorphic to $\bar\rho$.
		\end{theorem}

		\begin{proof}
			Let $\ell$ be an odd prime divisor of $\rad'(abcxyz)$. Then $E[p]$ is finite at $\ell$
			if and only if $\ell \nmid abc$. Indeed,
			\[
				j_E=\dfrac{2^8 (C^2 - AB)^3}{A^2 B^2 C^2}
			\]
			and $v_\ell(j_E) =- 2 v_\ell(abc) -2pv_\ell(xyz)$.
			Thus $p\mid v_\ell(j_E)$ if and only if 
			\[
				p\mid v_\ell(abc).
			\]
			Recall that $p>v_\ell(abc)$ by assumption \eqref{primebound}. Thus $E[p]$ is finite
			at an odd prime $\ell$ if and only if $\ell\nmid abc$. 
			
			Ribet's level lowering Theorem states that $\bar\rho$ is modular of level $2^s\rad'(abc)$. 
			I.e., there is a newform in $S_2(M)$, for some $M\mid 2^s \rad'(abc)$, and 
			a prime $\mathfrak p\ni p$ such that $\bar{\rho}_{f,\mathfrak p}$ and $\bar\rho$
			are isomorphic. See \ref{levelcomputation} for a proof of the equality 
			$M=2^s\rad'(abc)$. 
%		
%			
%			Let $R$ be the largest divisor
%			of $N$ coprime to $2p$. Then, $\bar\rho$ is ramified at every prime divisor of $R$ by
%			Tate's uniformization. Since $\bar{\rho}_{f,\mathfrak p}$ is unramified at primes not dividing $pM$
%			one has that $R\mid M$. 
%			
%			If $p\mid N$ then $p\mid M$. Indeed, assume $p\mid N$, $p\nmid M$ then
%			$\bar{\rho}_{f,\mathfrak p}\vert_{G_p}$ is either irreducible or reducible and finite. 
%			Again by Tate's uniformization $\bar\rho\vert_{G_p}$ is reducible and not finite.
%			This leads to a contradiction since $\bar\rho\simeq \bar{\rho}_{f,\mathfrak p}$. 
%			
%			It remains to prove that $2^s\mid M$. If $r\in\{0,3,5\}$ then $\bar\rho$ has exponent conductor
%			$r$ over $\mathbb Q_2$. Hence $s=r$ and we are done. If $r=1$, then either $s=1$ and 
%			$\bar{\rho}_{f,\mathfrak p}$ is ramified at $2$, hence $2\mid M$ or $s=0$ 
%			
%			Let $s$ denote the conductor exponent of $\bar{\rho}\vert_{G_2}$.
%			As discussed in the proof to Lemma \ref{irredlemma} one has that $s\leq r$ with equality
%			if  $r\in\{0,3,5\}$. 
%		
%			Case $(s,r)=(0,1)$ corresponds to $v_2(ABC)\geq 5$ and $v_2(abc)=4$. 
%			Indeed, $r=1$ if and only if $v_2(ABC)\geq 5$. In that case $s=0$ if and only if $E[p]$
%			is unramified at $2$ if, and only if, $p\mid v_2(j_E)$.
%			Since 
%			$$
%				v_2(j_E)= 8 -2v_2(ABC) \equiv 8-2v_2(abc) \pmod p
%			$$
%			and $p>v_2(abc)$ by assumption \ref{primebound} then $s=0$ is equivalent to $v_2(abc)=4$. 
		\end{proof}
	The final step is to connect $f$ with $S$-unit equations.
	Kraus method allows us to impose the following conditions.  
	\begin{itemize}
		\item $[\mathbb Q_f:\mathbb Q]=1$ so that $f$ corresponds to an elliptic curve $E'/\mathbb Q$.
		\item $E'$ has full rational $2$-torsion.
	\end{itemize}
	
	\begin{theorem}\label{Krauscurves}
		There is a constant $H=H(\rad(abc),s)$ such that if $p > H$ 
		then the newform described in Theorem \ref{newform}
		corresponds to an elliptic curve over $\mathbb Q$ with
		full rational $2$-torsion (up to isogeny). 
	\end{theorem}
	\begin{proof}
		See Théorème 3 and Théorème 4 in \cite{Kraus}.
	\end{proof}
	
	Notice that $H$ depends on $[x:y:z]$ since $r$ and $s$ may vary from 
	point to point. Nevertheless, one can still give a uniform
	bound depending only on $a, b, c$ by taking 
	$\max_s(H(\rad(abc), s))$.
	
	\begin{proposition}\label{reciprocalFrey}
		Let $N$ be a square-free odd integer and let $r\in \{0,1,3,5\}$.  
		The existence of a Frey Curve of conductor $2^rN$ is equivalent
		to the existence of a proper $2N$-point of height 
		\[
			\begin{array}{cl}
				\geq 5				&\text{ if $r=1$,}\\
				4					&\text{ if $r=0$,}\\		
				\text{$2$ or $3$} 		&\text{ if $r=3$,}\\
				1					&\text{ if $r=5$.}						
			\end{array}
		\]
	\end{proposition}
	\begin{proof}
		The existence of a Frey curve attached to a $2N$-point follows as in Proposition \ref{toprove2}.
		For the other implication let 
		\[
			E:Y^2 = X(X-A)(X+B)
		\]
		be a Frey curve of conductor 
		$2^r N$, $A,B\in \mathbb Z$. There is a Frey curve 
		\[
			E' :Y^2 =X(X-a)(X+b)
		\] 
		twist of $E$ such that $a$, $b$ are coprime, 
		$a\equiv -1\mod 4$ and $b$ is even. Mainly, the tern $(a,b,-a-b)\in \mathbb Z^3$ is a primitive representative of 
		$[A:B:-A-B]$ up to permutation of coordinates. Let us see that $(a,b,c)$ is a proper $2N$ point with 
		the corresponding constrains on the height. 
		The curve $E'$ has conductor $2^{r'} \rad'(ab(a+b))$ where 
		\[
			r'= \left\{
			\begin{array}{cl}
				0	&\text{if $v_2(b) = 4$,}\\
				1	&\text{if $v_2(b) \geq 5$,}\\
				3	&\text{if $v_2(b) = 2,3$,}\\
				5	&\text{if $v_2(b) = 1$}
			\end{array}
			\right.
		\]
		as described in \cite{DiamondKramer}. Thus, it is enough to prove that $r'=r$ and
		$N=\rad'(ab(a+b))$. Let $g\in \mathbb Z$ square-free such that $E$ is a twist of $E'$
		by the quadratic character $\chi$ attached to $\mathbb Q(\sqrt g)$. Equivalently, 
		\[
			E\simeq E'':Y^2 = X(X-a g) (X+bg).
		\]
		This model of $E''$ is minimal over $\mathbb Z_p$ for every odd prime $p$. 
		Thus $E$ has additive reduction at every odd prime divisor of $g$. Since $N$ is square-free
		one deduces that $g\in \{\pm 1,\pm2\}$ and $N = \rad'(ab(a+b))$. The conductor of $E'$ over 
		$\mathbb Q_2$ needs special consideration. Assume $g\neq 1$, otherwise
		$E$  and $E'$ have common conductor. 
		The character $\chi$ has conductor $2^{|g|+1}$. We check $r=r'$ via modularity.
		Let $f$ be the weight $2$, level $2^{r'}N$, trivial character newform attached to $E'$, by Wiles.  Then 		$f\otimes \chi$ is the newform attached to $E$ and has level $2^r N$. 
		If $r'\neq 5$ or $g\neq-1$, then the level of $f\otimes \chi$ is $2^{2|g|+2}N$ by 
		\cite[Theorem 3.1]{AtkinLi}. This contradicts that $r\in\{0,1,3,5\}$.
		 If $r'=5$ and $g=-1$ then $f\otimes \chi$ has level $2^5N$. Indeed, if 
		$g= f\otimes \chi$ has level $2^s N$ with $s\leq 3$ then $g\otimes \chi = f$ has level $2^4N$,
		\emph{loc. cit}. This contradicts $r'\in\{0,1,3,5\}$.
		Hence, either $g=1$ or $g=-1$ and $r=r'=5$.
		
		One can also use Tate's algorithm \cite{Tate} to compute those conductors. 
		See also \cite[Proposition 1]{Ulmer} for the case where $\mathbb Q(\sqrt{g})$ is more deeply 
		ramified than $\rho_{E',\ell}$, i. e. $r'\leq 1$ and $g\neq 1$.
	\end{proof}

	\section{Kraus Theorem}
	The following is a slight generalization of Kraus' Theorem \cite[Théorème 1]{Kraus}. 
	That is, the assumption $a,b,c$ being pairwise coprime has been relaxed to $\boldsymbol{(F)}$.

	Let $a,b,c$ be non-zero integers satisfying condition $\boldsymbol{(F)}$ defined in Section 
	\ref{FermatType}.
	We assume without loss of generality that $a$ is odd. 
	\begin{theorem}[$2$-good]\label{theorem2}
		Assume that $b$ is odd and $v_2(c) = 4$. 
		Let $S=\rad(abc)$ and assume that there are no proper $S$-points of $L_0$ of height $4$.
		Then there is a constant $G(a, b, c)$ such that
		\[
			AF_{a, b, c}\subseteq\bigcup_{5\leq p\leq G(a, b, c)} F_p^{a, b, c} \cup \{\text{trivial points}\}.
		\]
	\end{theorem}
	
	\begin{theorem}[$2$-node]\label{theorem1}
		Assume one of the following
		\begin{enumerate}
			\item $bc$ is odd,
			\item $b$ is odd and $v_2(c) \geq 5$, or
			\item $v_2(b) = v_2(c) \geq 1$.
		\end{enumerate}
		Let $S=\rad(2abc)$ and assume that there are no proper $S$-points
		of height $\geq 5$. Then there is a constant $G(a, b, c)$ such that
		\[
			AF_{a, b, c}\subseteq \bigcup_{5\leq p \leq G(a, b, c)} F_p^{a, b, c}\cup\{\text{trivial points}\}.
		\]
	\end{theorem}
	
	\begin{theorem}\label{theorem3}
		Assume that $b$ is odd and $v_2(c) \in \{2,3\}$. 
		Assume that there are no proper $S$-points of height $2$, $3$ or $\geq 5$ for $S=\rad(abc)$. 
		Then there is a constant $G(a, b, c)$ such that
		\[
			AF_{a, b, c}\subseteq \bigcup_{5\leq p \leq G(a, b, c)} F_p^{a, b, c}\cup\{\text{trivial points}\}.
		\]
%		\begin{itemize}
%			\item either $AF_{a, b, c}\subseteq \bigcup_{5\leq p \leq G(a, b, c)} F_p^{a, b, c}\cup\{\text{trivial points}\}$,
%			\item or $AF_{a, b, c}$ has infinitely many primitive points $[x:y:2z]$
%			 and there are $S$-unit points of height $\geq 5$. 
%		\end{itemize}
	\end{theorem}
	\begin{theorem}
		Assume that $b$ odd, $v_2(c)=1$ and let $S=\rad(abc)$. Assume that there are no proper 
		$S$-points of height $1$ or $\geq 5$.
		Then there is a constant $G(a, b, c)$ such that
		\[
			AF_{a, b, c}\subseteq \bigcup_{5\leq p \leq G(a, b, c)} F_p^{a, b, c}\cup\{\text{trivial points}\}.
		\]
%		\begin{itemize}
%			\item either $AF_{a, b, c}\subseteq \bigcup_{5\leq p \leq G(a, b, c)} F_p^{a, b, c}\cup\{\text{trivial points}\}$,
%			\item or $AF_{a, b, c}$ has infinitely many primitive points $[x:y:2z]$ and 
%			there are $S$-unit points of height $\geq 5$.  
%		\end{itemize}
	\end{theorem}
	
	See \ref{sectionbounds} for a description of $G(a, b, c)$. 
	
	\begin{proof}
		Assume that there is a prime $p>G(a, b, c)$ such that $F_p^{a, b, c}(\mathbb Q)$ 
		has non-trivial points.
		Let $\alpha,\beta,\gamma$ be pairwise coprime integers such that 
		$F_p^{a, b, c} \simeq F_p^{\alpha,\beta,\gamma}$ and 
		$\rad(abc) = \rad(\alpha\beta\gamma)$ by Lemma \ref{lema1.3}
		and let $(x,y,z)$ be a primitive 
		representative of a non trivial point $P$ in $F_p^{\alpha,\beta,\gamma}(\mathbb Q)$.
		Let $(A, B, C) = (\alpha x^p , \beta y^p,\gamma z^p)$ and  reorder $(A, B, C)$
		so that $E=E_{A, B, C}$ has minimal conductor. If $b,c$ are both even then the choice of $G(a, b, c)$
		ensures that $v_2(\alpha)\geq 5$, see  Remark \ref{technicalvaluation}. Otherwise, 
		$v_2(\alpha\beta\gamma)=v_2(abc)$. The level lowering trick combined with the fact that $p$ is large,
		see Theorem \ref{Krauscurves},
		implies that there is an elliptic curve $E'$ over $\mathbb Q$ with full rational $2$-torsion such that
		$E[p] = E'[p]$. Moreover $E'$ has conductor $2^s\rad'(abc)$ where
		\[
			s=
			\begin{cases}
				0		&\text{if $b$ is odd and $v_2(c) = 4$,}\\
				1		&\text{if $b, c$ have same parity,}\\
				1		&\text{if $b$ is odd, $v_2(c)\in\{1,2,3\}$ and $xyz$ is even,}\\
				3		&\text{if  $b$ is odd, $v_2(c) \in \{2,3\}$ and $xyz$ is odd,}\\
				5		&\text{if  $b$ is odd, $v_2(c) = 1$ and $xyz$ is odd.}
			\end{cases}
		\]
		Thus $E'=E_{R,S,T}$ is a Frey curve, with $\rad(RST) = \rad(2abc)$. That is, there
		is a proper $\rad(2abc)$-point of height
		\[
			\begin{cases}
				1				&\text{if $s=5$,}\\
				\text{$2$ or $3$}	&\text{if $s=3$,}\\
				4		&\text{if $s=0$,}\\
				\geq 5	&\text{if $s=1$,}
			\end{cases}
		\]
		by Proposition \ref{reciprocalFrey}.
	\end{proof}
	
	\section{Statements}\label{statements}
	In this section we translate Lemmas \ref{mod24}, \ref{r4}, \ref{r41} and \ref{r2}  
	to new cases of Asymptotic Fermat Conjecture with coefficients. 
	Let $(a, b, c)$ be a tern of non-zero integers satisfying $\boldsymbol{(F)}$,  $a$ odd. 

	\begin{theorem}
		Let $S$ be a set of primes all in $ 1+3\mathbb Z$.
		and assume that $\rad(abc) = S$.
		Then the Fermat equation 
		\[
			ax^p + b y^p + 16cz^p=0
		\]	
		has no solutions other than $xyz=0$ for $p$ larger than $G(a,b,16c)$.
	\end{theorem}
	\begin{proof}
		There are no proper $2S$-points of height $4$ by Lemma \ref{r4}.
		The theorem follows due to Theorem \ref{theorem2}. 
	\end{proof}	
	
	\begin{theorem}
		Let $n$ be a positive integer not dividing $14,16,18$ and let $S$ be a finite
		set of primes all in $(1 + n\mathbb Z)\cup (-1+ n\mathbb Z)$.
		Assume $\rad(abc)=S$. Then the Fermat equation
		\[
			ax^p + b y^p + 16 c z^p=0
		\]
		has no solutions other than $xyz=0$ for $p$ larger than $G(a,b,16c)$.
	\end{theorem}
	\begin{proof}
		There are no proper $2S$-points of height $4$ by Lemma \ref{r41}.
		The theorem follows due to Theorem \ref{theorem2}.
	\end{proof}
	
	\begin{theorem}
		Let $q$ be an odd prime and let $S$ be a finite set of primes all in  $1+4q\mathbb Z$.
		Assume that either $\rad(abc) = S$ or $\rad(abc)=2S$ and $v_2(bc) \geq 2$. Then the 
		Fermat equation 
		\[
			ax^p + by^p + c z^p =0
		\]
		has no solutions other than $xyz=0$ for $p$ larger than $G(a, b, c)$.
	\end{theorem}
	\begin{proof}
		Case $\rad(abc) = S$ follows from Theorem \ref{theorem1} and Lemma \ref{r2}.
		If $v_2(bc)\geq 2$ then either $b$ and $c$ are even or $b$ is odd and 
		$m=v_2(c) \geq 2$.
		The first case follows by Theorem \ref{theorem1} and Lemma \ref{r2}. 
		The second case follows by Theorem \ref{theorem1} for $m\geq 5$, by Theorem 
		\ref{theorem2} for $m=4$  and by Theorem \ref{theorem3} for $2\leq m\leq 3$.
	\end{proof}
	
	\begin{theorem}
		Let $q$, $\ell\geq 5$ be primes.
		Assume one of the following:
		\begin{enumerate}
			\item $(q,\ell) \equiv (-5, 5)$ or $(11,-11) \pmod {24}$.
			\item $q\equiv 11 \pmod{24}$, $\ell\equiv 5\pmod{24}$ and
			$\genfrac (){}{}{q}{\ell} = -1$.
			\item $q\equiv \pm 3 \pmod 8$, $\ell\equiv -1 \pmod {24}$, $\ell\not 
			\equiv -1
			\pmod {q}$.
		\end{enumerate}

		Assume that $\rad(abc) = \ell q$. 
		
		Let $n=0$ or $\geq 4$ then the Fermat equation
		\[
			ax^p + by^p + 2^ n c z^p =0,
		\]
		has no solutions other than $xyz=0$ for $p$ larger than $G(a, b,2^n c)$.
		
		Let $r\geq 1$ then the Fermat equation
		\[
			ax^p + 2^r by^p + 2^ r c z^p =0,
		\]
		has no solutions other than $xyz=0$ for $p$ larger than $G(a, 2^rb,2^r c)$.
	\end{theorem}
	\begin{proof}
		If either $n=0$ or $n\geq 5$ or $r\geq 1$ then this is Theorem \ref{theorem1} with Lemma \ref{mod24}.
		If $n=4$ then this is Theorem \ref{theorem2} with Lemma \ref{mod24}.	
	\end{proof}

	\section{Bounds}\label{sectionbounds}
		Let us recall the explicit bound $G(a, b, c)$ as in Kraus' paper. 
		In the following presentation we relax the bound so that statements are shorter.
		For example, $G(a, b, c)$ is taken so that $a, b, c$ are $p$th-power-free for every $p>G(a, b, c)$.
		
		Let us describe the bound $G(a, b, c)$. Let $N$ be a positive integer and let 
		\begin{align*}
			\mu(N) 	&= [\SL_2(\mathbb Z) : \Gamma_0(N)]=N\cdot \prod_{\text{$\ell\mid N$ prime}}\left(1+\dfrac{1}{\ell}\right),\\
			g(N)		&=\dim_{\mathbb C} S_2^{new}(N),\\
			F(N)		&=\left(\sqrt{\dfrac{\mu(N)}{6}}+1\right)^{2g(N)}.
		\end{align*}
		where $S_2^{new}(N)$ denotes the space of weight $2$ newforms of
		level $N$. 
		Let $a, b, c$ be non-zero integers satisfying $\boldsymbol{(F)}$, 
		$0=v_2(a)\leq v_2(b)\leq v_2(c)$. Let
		\[
		N=
			\begin{cases}
				\rad'(abc)		&\text{if $b$ is odd and $v_2(c)=4$,}\\
				2^3 \,\rad'(abc)	&\text{if $b$ is odd and $v_2(c) = 2,3$,}\\
				2^5\, \rad'(abc)	&\text{if $b$ is odd and $v_2(c)=1$,}\\
				2 \,\rad'(abc)	&\text{otherwise.}\\
			\end{cases}
		\]
		
		If $b$ is odd then $G$ is defined by
		\[
			G(a, b, c) :=\max(F(N), \max_{\text{$q$ prime}}v_q(a),\max_{\text{$q$ prime}}v_q(b),\max_{\text{$q$ prime}}v_q(c)).
		\]
		If $b$ is even, that is $v_2(b) = v_2(c)\geq 1$ then $G$ is defined by
		\[
			G(a, b, c) :=\max(F(N), \max_{\text{$q$ prime}}v_q(a),\max_{\text{$q$ prime}}v_q(b),\max_{\text{$q$ prime}}v_q(c), v_2(c) + 4).	
		\]
		\begin{example}
			Let $S\neq \emptyset$ be a finite set of primes in $1+12\mathbb Z$ and let $a, b, c$ 
			be non-zero, square-free, pairwise coprime integers 
			such that $\rad(abc) = S$. Then 
			\begin{align*}
				N&=2\rad(abc)=2S,\\
				g(N) &= \dfrac{\varphi(S)}{12} + (-1)^{\omega(2S)},\\
				\mu(N) &= 3\prod_{\ell\in S} (\ell+1).
			\end{align*}	
			Here $\varphi$ denotes the Euler's totient function and $\omega(2S)$
			the number of prime divisors of $2S$.
			The dimension $g(N)$ of $S_2^{new}(N)$ has been computed in \cite{Martin}.
		\end{example}

\appendix
\section{Prime divisors of cyclotomic polynomials}\label{appA}

	In this appendix we give some lower bounds for the number of prime divisors of 
	$\ell^n\pm 1$ for integers $\ell \geq 3$ and $n\geq 1$. 
	
	Let $\Phi_n$ be the $n$th cyclotomic polynomial. 
	A usual description of $\Phi_n$
	is given by the formula
	\[
		\Phi_n(X) =\prod_k (X-\zeta_n^k)
	\]
	where $\zeta_n = e^{2\pi i/n}$ is a primitive $n$th root of unity and $k$ ranges over 
	the units of $\mathbb Z/n\mathbb Z$. 
	%That is, $\Phi_n$ is the separable polynomial	with roots consisting in order $n$ elements of $\mathbb C^\times$.
		Gauss proved that $\Phi_n$ is irreducible in $\mathbb Z[X]$, hence $\mathbb Z[X]/\Phi_n \simeq \mathbb Z[\zeta_n]
		\subseteq \mathbb C$ is a domain. In particular
		\begin{equation}\label{factors1}
			X^n-1 = \prod_{d\mid n} \Phi_d(X)
		\end{equation}
		is the factorization of $X^n-1$ in irreducible factors over $\mathbb Z[X]$. 
		Similarly, write $n = 2^mn_2$ where $n_2$ is the largest odd divisor of $n$. 
		Then
		\begin{equation}\label{factors}
			X^n+1 = \prod_{d\mid n_2} \Phi_{2^{m+1}d} (X)
		\end{equation}
		since $X^{2^{m+1}n_2}-1 = (X^n-1) (X^n+1)$. 
%	Also if $d\mid n$ then $\mathbb Z[\zeta_d]\rightarrow \mathbb Z[\zeta_n]$
%	$\zeta_d\mapsto \zeta_n^{n/d}$ defines an inclusion.

%	In this section we are interested in the factorization of $\Phi_n(\ell)$. 
%	The map $\mathbb Z[\zeta_n]\rightarrow A=\mathbb Z/\Phi_n(\ell)$,
%	$\zeta_n\rightarrow \ell$ defines a surjective ring homomorphism.
	
%	Consider the projection 
%	$\pi_{\ell,p}: \mathbb Z[X] \rightarrow \mathbb F_p$ given by $\pi_\ell(X) = \ell\in \mathbb F_p$.
%	Then $\pi_{\ell,p}$ factors through $\mathbb Z[\zeta_n]$ if and only if $p\mid \Phi_n(\ell)$. 
	
	Let $k$ be a positive integer. The map $\mathbb Z[X]\rightarrow \mathbb Z/k\mathbb Z$, $X\mapsto \ell$
	factors through $\mathbb Z[\zeta_n] \rightarrow \mathbb Z/k\mathbb Z$, $\zeta_n\mapsto \ell$
	if, and only if, $k\mid \Phi_n(\ell)$. 
		
	\begin{lemma}\label{lemaA.1}
		Let $p\nmid n$ be a prime and assume that there is a ring homomorphism
		$\theta:\mathbb Z[\zeta_n] \rightarrow \mathbb F_p$. Then $\theta(\zeta_n)$ has order $n$ in $\mathbb F_p^\times$ and $n\mid p-1$. 
	\end{lemma}
	\begin{proof}
		Let $\alpha = \theta (\zeta_n)$. Then $\alpha^n-1 = \prod_d \Phi_d(\alpha)=0$. Notice that $X^n-1$ 
		is separable over $\mathbb F_p$ since $nX^{n-1}\neq 0$ in $\mathbb F_p[X]$.
		Hence $\alpha$ has order $n$ in $\mathbb F_p^\times$ and the lemma follows. 
	\end{proof}

	\begin{lemma}\label{boundvaluation}
		Let $p$ be an odd prime. There is no ring homomorphism $\mathbb Z[\zeta_p] \rightarrow \mathbb Z/p^2\mathbb Z$. 
		There is no ring homomorphism $\mathbb Z[\zeta_4] \rightarrow \mathbb Z/4\mathbb Z$. 
	\end{lemma}
	\begin{proof}
		It is enough to prove that $\Phi_p(X) = \sum_{i=0}^{p-1} X^i$ has no roots in $\mathbb Z/p^2\mathbb Z$. 
		The following proof is standard.
		Assume that there is a root $a$ of $\Phi_p$ in $\mathbb Z/p^2\mathbb Z$. Then $a = 1 \mod p$, since 
		$\Phi_p = (X-1)^{p-1}$ in $\mathbb F_p$. Notice that $\Phi_p(1 + pb ) = \sum_{i=0}^{p-1} 1 + i pb=p$
		in $\mathbb Z/p^2\mathbb Z$ for every $b$. Hence $\Phi_p(a) = p$ for every $a\equiv 1\pmod p$. 
		
		Notice that $\Phi_4(X)=X^2 + 1$ has no roots in $\mathbb Z/4\mathbb Z$. 
	\end{proof}

%	The polynomial $\Phi_n$ is separable over $\mathbb F_p$ if and only if $p\nmid n$. 
%	There are $[\mathbb F_p(\zeta_n):\mathbb F_p]$  homomorphisms $A_n\rightarrow \bar{\mathbb F_p}$ 
%	is isomorphic to $(\mathbb Z/n\mathbb Z)^\times$. 

	\begin{lemma}\label{lemaA.3}
		Let $\ell\geq 3$, $n\geq 2$ be integers and let $p$ be the largest prime divisor of $n$, 
		then $|\Phi_n(\ell)|> p$.  
	\end{lemma}
	\begin{proof}
		The Euler's totient function $\varphi$, 
		%$\varphi(n):=\deg \Phi_n$,
		satisfies that 
		\[
			p-1\mid \varphi(n).
		\]
%		The restriction of $\mathbb Z/n\mathbb Z\rightarrow \mathbb F_p$ to groups of units
%		is a surjective group homomorphism. In particular $\deg \Phi_n\geq p-1$.
		Hence 
		\[
		| \Phi_n(\ell) |= \prod_k |\ell-\zeta_n^k| \geq \prod_k 2 \geq  2^{p-1}
		\]
		and case $p\geq 3$ follows. 
		
		If $p=2$ then $n$ is a power of $2$, $n= 2^m$, and 
		\[
			\Phi_n(\ell) = \ell^{2^{m-1}}+1> 2.
		\]
	\end{proof}
	The polynomial $\Phi_n$ has no real roots for $n\geq 3$, hence $|\Phi_n(\ell)| =\Phi_n(\ell)$.
	
	\begin{theorem}\label{1805}
		Let $\ell \geq 3, n\geq 3$ be integers. 
		There is a prime divisor $p$ of $\Phi_n(\ell)$ not dividing $2n$. Hence, $\ell$ has order $n$
		in $\mathbb F_p^\times$. 
	\end{theorem}
	\begin{proof}
		\emph{Case $n=2^m\geq 4$.}\\
		One has that $\Phi_{2^m}(X)= X^{2^{m-1}}+ 1$ and $\Phi_n(\ell) \geq 10$. 
		If $4\mid \Phi_n(\ell)$ then $\mathbb Z[\zeta_n]\rightarrow \mathbb Z/4\mathbb Z$,
		$\zeta_n\mapsto \ell$ defines a ring homomorphism that restricts to $\mathbb Z[\zeta_4]\subseteq
		\mathbb Z[\zeta_n]$. 
		This contradicts Lemma \ref{boundvaluation}. Hence either $\Phi_n(\ell)$ is odd or $\Phi_n(\ell)/2\geq 5$ is odd.
		
		\emph{Case $p\mid n$, $p$ odd.}\\
		Notice that $\Phi_n(\ell)$ is odd. Indeed, if $2\mid \Phi_n(\ell)$ then there exists a ring homomorphism 
		\[
			\mathbb Z[\zeta_n]\rightarrow \mathbb F_2
		\]
		which induces by restriction a map 
		\[
			\mathbb Z[\zeta_p]\rightarrow \mathbb F_2
		\]
		hence $p\mid 2-1$ by Lemma \ref{lemaA.1}. 
		
		Let us see that either $\Phi_n(\ell)$ and $n$ are coprime 
		or there is a prime $p$ such that $\Phi_n(\ell)/p, n$ are coprime. Assume that $p<q$
		are prime divisors of $\Phi_n(\ell)$ and $n$.
		Then there is a ring homomorphism
		\[
			\mathbb Z[\zeta_q] \subseteq \mathbb Z[\zeta_n] \rightarrow \mathbb F_p
		\]
		and $q\mid p-1$ by Lemma \ref{lemaA.1} which contradicts $p<q$. Hence the greatest
		common divisor of $\Phi_n(\ell)$ and $n$ is a possibly trivial power of an odd prime $p$.
		If $p\mid n$ then $p^2\nmid \Phi_n(\ell)$ by Lemma \ref{boundvaluation}. 
		Hence either $\Phi_n(\ell), 2n$ are coprime or there is an odd prime divisor $p$
		of $\Phi_n(\ell)$ such that $\Phi_n(\ell)/p$ and $2n$ are coprime. In the second case
		$\Phi_n(\ell)/p$ is an odd integer $>1$ by Lemma \ref{lemaA.3} and the first part of the theorem 
		follows. The order of $\ell$ is computed in Lemma \ref{lemaA.1}.
	\end{proof}
	
	\begin{corollary}\label{primedivisors}
		Let $\ell \geq 3$. Assume that $n_1, \dots, n_r$ are pairwise different integers $\geq 3$. 
		Then 
		\[
		\prod_i \Phi_{n_i}(\ell)
		\]
		has at least $r$ odd prime divisors.
	\end{corollary}

	\begin{proof}
		Let $p_i$ be a prime divisor of $\Phi_{n_i}(\ell)$ coprime to $2n_i$ as in Theorem \ref{1805}. 
		Then $\ell$ has order $n_i$ in $\mathbb F_{p_i}^\times$, thus $p_i\neq p_j$ for different $i, j$.
	\end{proof}
	
	For an integer $k$ let $\omega(k)$ denote the number of prime divisors of $k$
	and let $\sigma(k)$ denote the number of divisors of $k$.
	
	\begin{corollary}\label{corminus}
		Let $\ell\geq 3$, $n\geq 1$ be integers. If $(\ell, n)\neq (3,even)$ then
		\[
			\omega(\ell^n-1)\geq \sigma(n).
		\]
		Otherwise
		\[
			\omega(3^{2t}-1)\geq \sigma(2t)-1.
		\]

	\end{corollary}
	
	\begin{proof}
		Let $i\in \{1,2\}$ such that $n \equiv i \pmod 2$. Then
		\[
			A:=\prod_{\substack{\text{$d\mid n$}\\ \text{$d\geq 3$}}}\Phi_d(\ell)=\dfrac{\ell^n-1}{\ell^i-1}
		\]
		has at least $\sigma(n)-i$ odd prime divisors $S=\{p_d\}_{d\mid n, d\geq 3}$ as in 
		Theorem \ref{1805}.
		Notice that $p_d\nmid \ell^i-1$ for every $p_d\in S$. Indeed, if an odd prime $p$ divides
		$\ell^i-1$ then $\ell$ has order $\leq i$ in $\mathbb F_{p}^\times$ by Lemma \ref{lemaA.1}. Thus 
		\[
		\omega(\ell^n-1) \geq  \sigma(n) - i + \omega(\ell^i-1).
		\]
		
		It is enough to prove that $\omega(\ell^i-1) \geq i$ if and only if $(\ell,i) \neq (3,2)$. 
		If $i=1$ then $\ell-1\geq 2$ and $\omega(\ell-1) \geq 1$. If $i=2$ then $gcd(\ell-1,\ell+1)\leq 2$.
		Assume $\omega(\ell^2-1) < 2$ then $\ell-1, \ell+1$ are powers of two. Hence $\ell=3$. 
	\end{proof}
	
%	Let $t$ be a positive odd integer and let $m$ be positive. The roots of $X^{2^m t}+1$
%	are the $t2^{m+1}$th roots of unity that are not $t2^{m}$th roots of unity. Hence
%	\begin{equation}\label{factors}
%		X^{2^mt}+1 = \prod_{d\mid t} \Phi_{2^{m+1} d}(X).
%	\end{equation}
	\begin{corollary}\label{corplus}
		Let $\ell\geq 3$, $n\geq 1$ be integers and let $n_2$ be the largest
		odd divisor of $n$. Then
		\[
			\omega(\ell^n+1)\geq \sigma(n_2).
		\]\
	\end{corollary}

	\begin{proof}
	Let $n=2^m n_2$ then $\ell^n+1 = \prod_{d\mid n_2} \Phi_{2^{m+1} d}(\ell)$ by 
	the polynomial factorization of \eqref{factors}. 
	For every $d$ such that $2^{m+1} d\neq 2$ consider a prime $p_d\mid \Phi_{2^{m+1} d}(\ell)$ 
	as in Theorem \ref{1805}. 
	If $m=0$ let $p_1$ be an arbitrary prime divisor of $\Phi_{2}(\ell)=\ell+1$. 
	Then $\prod_{d\mid n_2} p_d$ is a squarefree divisor of $\ell^n+1$. 
	\end{proof}	
	
	\subsection{Catalan Conjecture}
	
	One deduces a case of Catalan's Conjecture.

	\begin{theorem}[Partial Catalan's Conjecture]\label{partialcatalan}
		Let $\ell\geq 3$ be an integer and assume that 
		\[
			2^m - \ell ^n \in \{\pm1\}
		\]
		for some integers $m, n\geq 2$. Then $m=\ell=3$, $n=2$.
	\end{theorem}
	\begin{proof}
		Assume that $2^m = \ell^n+1$, $n\geq 2$ and let $n_2$ be the largest odd divisor of $n$.
		Then $\ell$ is odd and $\ell^n+1 \geq 4$. 
		By Corollary \ref{corplus} we have that $1=\omega(\ell^n +1)\geq \sigma(n_2)$, hence $n_2=1$ and
		$n=2^r$ for some positive $r$. Since $2^m=\ell^{2^r}+1 \equiv 2 \pmod 4$ one has that $m=1$ and 
		$2= \ell^{2^r} + 1$. 
		
		Assume that $2^m = \ell^n-1$, $n\geq 2$. If $(\ell,n)=(3, 2t)$ with $t$ an integer then
		$1=\omega(3^{2t}-1) \geq \sigma(2t)-1$ by Corollary \ref{corminus}. Hence $t=1$.
		
		If $(\ell, n)\neq (3, even)$, by Corollary \ref{corminus} one has that $1=\omega(\ell^n-1) \geq \sigma(n)$. 
		Hence $n=1$. 
	\end{proof}

	This partial result is well known to experts, see \cite[B3.3]{Ribenboim}. 
	See \emph{ibid} for a complete treatment of Catalan's conjecture written before 
	Preda Mih\u{a}ilescu's proof \cite{Mihailescu}. See also Bilu - Bugeaud - Mignotte's book \cite{Bilu} 
	for a \emph{minimalistic} approach of the proof or Schoof's book \cite{Schoof} based on two 
	sets of lecture notes by Yuri Bilu.

\section{The conductor of \texorpdfstring{$E[p]$}{E[p]}}\label{levelcomputation}
	The $j$-invariant of a Frey curve is given by the formula
		\[
			j_E=\dfrac{2^8 (C^2 - AB)^3}{A^2 B^2 C^2}.
		\]
	Thus one has for the case $(A,B,C) = (ax^p , by^p ,cz^p)$ being pairwise coprime that $C^2-AB$ and $ABC$ are coprime.
	Let $\ell$ be a prime divisor of $ABC$. Then
	\[
		v_\ell(j_E) = 8 v_\ell(2) - 2v_\ell(ABC) \equiv 8v_\ell(2) - 2v_\ell(abc)\pmod p.
	\]
	Thus $p\mid v_\ell(j_E)$ if and only if 
	\begin{itemize}
		\item $\ell$ is odd and $p\mid v_\ell(abc)$,  or
		\item $\ell=2$ and $v_2(abc) \equiv 4 \pmod p$. 
	\end{itemize}
	\begin{proposition}
		Let $E=E_{A,B,C}$ be the Frey curve as in Theorem \ref{newform}. Let $f$
		be  a newform in $S_2(M)$ for some divisor $M$ of $2^s\rad'(abc)$ and let $\mathfrak p$ be a prime ideal such that 
		\[
			E[p] \simeq \bar{\rho}_{f,\mathfrak p}
		\]
		as $\mathbb F_p[G_{\mathbb Q}]$-modules. Then $M= 2^s \rad'(abc)$.
	\end{proposition}
	\begin{proof}
		Let $R$ be the largest (square-free) divisor of $2^s\rad'(abc)$ coprime to $2p$.
		By Tate's uniformization $E[p]$ is ramified at every prime divisor $\ell$ of $R$
		and so is $\bar{\rho}_{f,\mathfrak p}$. Thus, $R\mid M$. 
		
		Let $\ell=2$. If $s \in\{ 3,5\}$ then Carayol \cite{Carayol89} predicts that the lifting 
		$\rho_{f,\mathfrak p}$ of $\bar{\rho}_{f,\mathfrak p}$ has 
		conductor exponent $s$. Thus $2^s\mid M$. 
		If $s=0$ then $M$ is odd and so is $R$. If $s=1$ then $E[p]$ is ramified at $2$ and so is $\bar\rho_{f,\mathfrak p}$.
		Hence $2\mid M$. 
		
		One could just avoid case $p\mid M$ since we will consider big primes $p$ with respect to $\rad(abc)$.
		Still, if $p\mid \rad'(abc)$ then $E[p]$ is not finite at $p$. That is, $E[p]\vert_{G_p}$ is reducible and not 
		\emph{peu ramifié} by \cite[Proposition 8.2]{Edixhoven}. If $p\nmid M$ then $\bar{\rho}_{f,\mathfrak p}\vert_{G_p}$
		is either irreducible or reducible and \emph{peu ramifié}. Thus $E[p]\vert_{G_p}\not\simeq \bar{\rho}_{f,\mathfrak p}\vert_{G_p}$. 
		This completes the proof.
	\end{proof}

	\section{Mod \texorpdfstring{$24$}{24} exercises}\label{mod24ex}
	\begin{proof}[Proof of Lemma \ref{mod24}:]
		Let $(A,B,C)$ be a primitive $S$-unit point of height $\geq 3$. Assume $A=2^r$, $r\geq 3$.
		Then $B+C\equiv 0\pmod 8$ and $B+C\not\equiv 0\pmod 3$. Hence,
		\[
			BC\equiv -1 \pmod 8
		\]
		since $C^{-1} \equiv C \pmod 8$
		and
		\[
			BC\equiv 1 \pmod 3
		\] 
		since $B,C\in \{\pm 1\} \mod 3$.
		Thus
		\[
			\pm q^s \ell^t = BC  \equiv 7 \pmod{24}.
		\]
		
		\begin{enumerate}
			\item 			
			By hypothesis $(q,\ell) \equiv (-5, 5)$ or $(11,-11)\pmod{24}$. 
			Notice that  
			\[
				q^s \ell^t\equiv \pm q^{s+t}\not\equiv \pm 7  \pmod{24},
			\] 
			hence $A$ is not a power of two.
			
			Assume that 
			\[
				0\equiv 2^r q^{s} = \ell^t + \varepsilon\equiv (- 3)^t +\varepsilon\pmod 8
			\] 
			for some $\varepsilon\in \{\pm1\}$. 
%			The equation $\ell^{t}+\varepsilon \equiv 0 \mod 8$ implies 
			Then $\varepsilon=-1$ and $t$ is
			even. Proposition \ref{completesol} implies 
			\[
				(q,\ell) \in\{(3,5),(5,3),(3,7),(3,17)\}.
			\]
			Condition $q\equiv -\ell\pmod{24}$ leads to a contradiction. Similarly, $2^r \ell^t = q^r+\varepsilon$ has no solution.

			\item Assume that $(2^r,- q^s \ell^t, \varepsilon)$ 
			is an $S$-point for some unit $\varepsilon$. Then 
			$- \varepsilon q^s \ell^t\equiv 7\pmod{24}$. Thus $s,t$ are odd and $\varepsilon = -1$.
			That is
			\[
			2^r = q^s \ell^t + 1\equiv -1\pmod 3,
			\]
			hence $r$ is odd, $r=2f+1$. Thus,
			$2$ is a square in $\mathbb F_q$, i.e. $q\equiv \pm 1\pmod 8$. 
			Indeed 
			\[
			\genfrac (){}{0}{1}{q}=\genfrac (){}{0}{2}{q}^r = \genfrac (){}{0}{2}{q}.
			\]
			Assume that $2^r +(-1)^{a}q^{s} +(-1)^{b}\ell^{t}=0$.
			Then 
			\[
				(-1)^{a+b} q^s \ell^t \equiv 7\pmod{24}.
			\] 
			Hence $a,b$ have same parity and $s,t$ are odd.
			Thus 
			\[
				2^r = q^s + \ell^t\equiv 1\pmod 3
			\] 
			and $r$ is even. Thus $q$ is a square in $\mathbb F_\ell$.
			
			Assume that $(2^r q^s , -\ell^t , \varepsilon)$ is an $S$-point. 
			Then $\ell^t \equiv \varepsilon \pmod 8$ and hence $t$ is even and $\varepsilon =1$.
			
			Assume that $(2^r \ell^t , -q^s , \varepsilon)$ is an $S$-point. Then $\varepsilon =1$
			and $s$ is even. By Proposition \ref{completesol} $q\in \{ 3,5,7,17\}$, 
			hence 
			\[
				q\not \equiv 11 \pmod{24}.
			\]
			
			\item By hypothesis 
			\[
				\ell \equiv -1 \pmod{24}
			\] 
			and $q\equiv \pm 5$ or $\pm 11 \pmod{24}$ since $q \geq 5$. 
			Thus $q^s \ell^t\not\equiv  \pm 7 \pmod{24}$ and $A$ is not a power of two. 
			
			Assume that $2^r q^{s} = \ell^{t} +1$.
			Then $t$ is either $1$ or an odd prime by Lemma \ref{technicallemma3}. 
			Case $t=1$ implies $\ell \equiv -1 \pmod q$. Case $t$ odd prime implies
			$\ell$ Mersenne hence 
			\[
			\ell\equiv0, 1\pmod 3.
			\]

			Assume that $2^r q^{s}=\ell^{t} -1$.
			Hence $t$ is even and Proposition
			\ref{completesol} implies $\ell \in \{3,5,7, 17\}$, then $\ell\not\equiv -1\pmod{24}$.
			Similarly, case $2^r \ell^s = q^t\pm 1$ is not allowed by Lemma \ref{completesol}.
		\end{enumerate}
	\end{proof}

\end{document}